\documentclass[11pt]{article}
\usepackage{amsthm, amsmath, amssymb, amsfonts, url, booktabs, tikz, setspace, fancyhdr, bm}

\usepackage{amsthm}

\usepackage{mathrsfs}
\usepackage{xcolor}
\usepackage[margin = 2.3cm]{geometry}
\usepackage{hyperref, enumerate}
\usepackage[shortlabels]{enumitem}
\usepackage[babel]{microtype}
\usepackage[english]{babel}
\usepackage[capitalise]{cleveref}
\usepackage{comment}
\usepackage{bbm}
\usepackage{csquotes}
\usepackage{mathabx}
\usepackage{tikz}
\usetikzlibrary{arrows.meta,decorations.pathreplacing,positioning}
\usepackage{graphicx}
\usepackage{float}
\usepackage{thm-restate}
\usepackage{mathtools}

\counterwithin{figure}{section}
\newcounter{shared}
\newtheorem{conjecture}[shared]{Conjecture}

\newtheorem{theorem}{Theorem}[section]
\newtheorem*{thm-non}{Theorem}
\newtheorem{prop}[theorem]{Proposition}
\newtheorem{lemma}[theorem]{Lemma}
\newtheorem{cor}[theorem]{Corollary}
\newtheorem{claim}[theorem]{Claim}

\theoremstyle{definition}

\newtheorem*{defn-non}{Definition}

\newtheorem{rmk}[theorem]{Remark}

\newlist{Case}{enumerate}{2}
\setlist[Case, 1]{%
    label           =   {\bfseries Case \arabic*.},
    labelindent=1em ,labelwidth=1.3cm, labelsep*=1em, leftmargin =!
}
\setlist[Case, 2]{%
    label           =   {\bfseries Subcase \arabic{Casei}.\arabic*.},
    labelindent=-1em ,labelwidth=1.3cm, labelsep*=1em, leftmargin =!
}

\newenvironment{poc}{\begin{proof}[Proof of the claim]}{\end{proof}}

\newcommand{\floor}[1]{\lfloor #1\rfloor}

\usepackage{todonotes}

\newcommand{\eps}{\varepsilon}

\newcommand{\chithr}{\delta_{\chi}}
\newcommand{\homthr}{\delta_{\textup{hom}}}
\newcommand{\dist}{\operatorname{dist}}
\newcommand{\hto}{\xrightarrow{\textup{hom}}}
\newcommand{\nhto}{\mathrel{\overset{\textup{hom}}{\nrightarrow}}}

\title{The homomorphism threshold of odd cycle \(C_{2k-1}\) is below \(\frac{1}{2k-1}\)}
\author{
Jian Wang\thanks{School of Mathematics, Sichuan University,
Chengdu, China. Email: wangjianmath01@scu.edu.cn. Jian Wang is supported by
National Natural Science Foundation of China Grant no. 12471316 and Natural Science Foundation of Shanxi Province Grant no. RD2500002993.}
\and
Shipeng Wang\thanks{School of Mathematical Sciences, Jiangsu University, Zhenjiang, China. Email: 
spwang22@ujs.edu.cn. Shipeng Wang is supported by
National Natural Science Foundation of China Grant no. 12001242.}
\and
Zixiang Xu\thanks{School of Mathematical Sciences, Zhejiang University, Hangzhou, China. Email: zixiangxu@zju.edu.cn.}}
\date{}

\begin{document}
\maketitle

\begin{abstract}
The homomorphism threshold \(\delta_{\textup{hom}}(H)\) of a graph \(H\) asks how large the minimum degree of an \(H\)-free graph has to be in order to force a homomorphism to a bounded \(H\)-free graph. Determining this threshold is in general very difficult, and the odd cycles \(C_{2k-1}\) are among the most important open cases for \(k\ge3\). Ebsen and Schacht proved the general upper bound \(\delta_{\textup{hom}}(C_{2k-1})\le\frac{1}{2k-1}\), while a breakthrough of Sankar, using topological methods and a graph-theoretic analogue of homotopy equivalence, gave the first positive lower bound. The value \(\frac{1}{2k-1}\) appeared particularly compelling: Ebsen and Schacht obtained the same exact threshold when all odd cycles of length at most \(2k-1\) are forbidden, Huang, Liu, Rong and Xu later proved that it is the exact blowup threshold of \(C_{2k-1}\), and Letzter and Snyder also explicitly asked whether \(\delta_{\textup{hom}}(C_{5})=\frac{1}{5}\). Surprisingly, we show that the upper bound can be improved. More precisely, for every integer \(k\ge3\), we prove
\[
\frac{1}{2\left((k-1)^{4k-5}(2k-1)+\frac{(k-1)^{4k-5}-1}{k-2}\right)}
\le
\delta_{\textup{hom}}(C_{2k-1})
\le
\frac{4(k-1)}{4(k-1)(2k-1)+1}
<
\frac{1}{2k-1}.
\]
The new lower bound comes from a new graph-theoretic construction based on a sparse homomorphism theorem of Ne\v{s}et\v{r}il and Zhu, and it improves Sankar's quantitative bound. The improved upper bound follows from a new structural argument that controls common neighborhoods along short odd paths. Our results have various consequences, in particular, every odd cycle of length at least five has pairwise distinct chromatic, homomorphism, polynomial removal, and linear removal thresholds, resolving two conjectures of Fox and Wigderson.
\end{abstract}

\section{Introduction}
An important way to measure how much minimum degree forces structure in \(H\)-free graphs is the chromatic threshold. 
An old result of Andr\'{a}sfai, Erd\H{o}s, and S\'{o}s~\cite{1974ErdosSos} showed that every $K_{s}$-free graph \(G\) with minimum degree larger than $\frac{3s-7}{3s-4}\cdot |V(G)|$ has chromatic number at most $s-1$, which further leads to the following general problem proposed by Erd\H{o}s and Simonovits~\cite{1973ErdosS}. For a given graph \(H\), define its chromatic threshold \(\chithr(H)\) to be the infimum of all \(\alpha\) such that every \(H\)-free graph \(G\) with \(\delta(G)\ge\alpha |V(G)|\) has bounded chromatic number. This problem was widely studied~\cite{2011JGT,2010ColoringViaVCDim,2010arxivKrfree,2002Thomassen,2007CombC5}, and was solved for all graphs by Allen, B\"ottcher, Griffiths, Kohayakawa and Morris~\cite{2013Adv}. We also refer the readers to some recent variants~\cite{GLWX2026,KLSWWX2025,NWX2026}.

Since having bounded chromatic number is equivalent to admitting a homomorphism to a clique of bounded order, Thomassen~\cite{2002Thomassen} asked for a stronger version in which the bounded image is also required to avoid \(H\). We write \(G\hto F\) if there is a homomorphism from \(G\) to \(F\). The homomorphism threshold \(\homthr(H)\) is the infimum of all \(\alpha\) such that for every \(\eps>0\) there is a constant \(b\) with the following property: every sufficiently large \(H\)-free graph \(G\) with \(\delta(G)\ge(\alpha+\eps)|V(G)|\) maps to an \(H\)-free graph on at most \(b\) vertices. Clearly \(\homthr(H)\ge\chithr(H)\). Equality is known for cliques: \L{}uczak~\cite{2006CombTriangle} proved the triangle case, and Goddard and Lyle~\cite{2011JGT} proved the general clique case, also see some new proofs and discussions in~\cite{2024GraphToGeom,2020CPCProb}. 

Odd cycles are the main case where this stronger question already shows a real difference. Answering an old question in~\cite{1973ErdosS}, Thomassen proved \(\chithr(C_{2k-1})=0\) for every \(k\ge3\)~\cite{2007CombC5}, so arbitrarily small positive minimum degree is enough to force bounded chromatic number in \(C_{2k-1}\)-free graphs. The homomorphism problem is harder because the bounded target must still be \(C_{2k-1}\)-free. Letzter and Snyder~\cite{2019JGTC3C5} showed that \(\homthr(C_{5})\le\frac{1}{5}\), and later Ebsen and Schacht~\cite{2020CombOdd} proved the general upper bound
\(
\homthr(C_{2k-1})\le\frac{1}{2k-1}
\)
for every \(k\ge3\).

For a family \(\mathcal{F}\) of graphs, one can define \(\homthr(\mathcal{F})\) analogously, requiring both the input graph and the bounded target to be \(\mathcal{F}\)-free. There were strong reasons to regard \(\frac{1}{2k-1}\) as the natural candidate for the exact value. Let
\(
\mathcal{C}_{2k-1}=\{C_{3},C_{5},\ldots,C_{2k-1}\}.
\)
In the same paper, Ebsen and Schacht~\cite{2020CombOdd} proved the exact equality
\(
\homthr(\mathcal{C}_{2k-1})=\frac{1}{2k-1}.
\)
More recently, Huang, Liu, Rong and Xu~\cite{2025Blowup} introduced the blowup threshold \(\delta_{\mathrm{B}}(H)\), which asks when every maximal \(H\)-free graph is a blowup of a bounded graph, and proved
\(
\delta_{\mathrm{B}}(C_{2k-1})=\frac{1}{2k-1}.
\)
This is a stronger structural requirement than merely having a bounded \(C_{2k-1}\)-free homomorphic image. Thus the same constant was sharp for the family of all odd cycles of length at most \(2k-1\) and for the stronger blowup problem, while it was also the best known upper bound for a single odd cycle. Taken together, these results made the equality \(\homthr(C_{2k-1})=\frac{1}{2k-1}\) look especially plausible. We also refer the interested readers to more related structural results on graphs without odd cycles of certain length~\cite{BFMCPS2023,GLWX2026,GHW2025,
Häggkvist1982,KLSWWX2025,
lu2026structure,MS2015,
NWX2026,
RWWY2024,YanPengYuan2024,yuan2024minimum}.

The first lower bound separating the chromatic and homomorphism thresholds is due to Sankar~\cite{2022sankar}, who proved \(\homthr(C_{2k-1})>0\) for every \(k\ge3\), using topological methods and a graph-theoretic analogue of homotopy equivalence. This was a breakthrough step because it showed that the restriction on the homomorphic image is not a small technical change. It nevertheless left an enormous quantitative gap to the candidate value \(\frac{1}{2k-1}\).

Much to our surprise, the natural candidate \(\frac{1}{2k-1}\) is not the correct upper bound. Our main result simultaneously gives a new graph-theoretic lower bound construction and a strict improvement on the Ebsen--Schacht upper bound~\cite{2020CombOdd}. 
\begin{theorem}\label{thm:main}
For every integer \(k\ge3\),
\[
\frac{1}{2\left((k-1)^{4k-5}(2k-1)+\frac{(k-1)^{4k-5}-1}{k-2}\right)}
\le
\homthr(C_{2k-1})\le
\frac{4(k-1)}{4(k-1)(2k-1)+1}
<
\frac{1}{2k-1}.
\]
Moreover, for every \(\eps>0\), every sufficiently large \(C_{2k-1}\)-free graph \(G\) satisfying
\[
\delta(G)\ge\left(\frac{4(k-1)}{4(k-1)(2k-1)+1}+\eps\right)|V(G)|
\]
admits a homomorphism to \(K_{2k-2}\).
\end{theorem}

The lower bound gives the bound of order \(k^{-4k+O(1)}\), improving the lower bound obtained from Sankar's construction~\cite{2022sankar}, which is of order \(k^{-(8+o(1))k^{2}}\). It is built from generalized Mycielski graphs and a sparse homomorphism theorem of Ne\v{s}et\v{r}il and Zhu~\cite{2004JCTBZhu}. The upper bound comes from a different mechanism. In particular, the upper bound argument does not merely replace the old target by another bounded graph, it even forces an ordinary \((2k-2)\)-coloring. For \(k=3\), the upper bound in \cref{thm:main} gives
\(
\homthr(C_{5})\le\frac{8}{41}<\frac{1}{5}.
\)
answering the question of Letzter and Snyder~\cite{2019JGTC3C5} in negative. 

We then obtain several consequences of \cref{thm:main}. Fox and
Wigderson~\cite{FoxWigderson} introduced the polynomial and linear
removal thresholds, denoted by
\(\delta_{\textup{poly-rem}}(H)\) and
\(\delta_{\textup{lin-rem}}(H)\), which measure the minimum degree
density above which the \(H\)-removal lemma admits, respectively,
polynomial and linear quantitative dependence. They made the following
two conjectures.
\begin{conjecture}[Fox--Wigderson~\cite{FoxWigderson}]
For every odd integer \(m\ge5\),
\(
\delta_{\textup{poly-rem}}(C_{m})>\homthr(C_{m}).
\)
\end{conjecture}

\begin{conjecture}[Fox--Wigderson~\cite{FoxWigderson}]
There exists a graph \(H\) for which
\(\chithr(H)\), \(\homthr(H)\),
\(\delta_{\textup{poly-rem}}(H)\), and
\(\delta_{\textup{lin-rem}}(H)\) are pairwise distinct.
\end{conjecture}

Both conjectures are completely resolved by odd cycles. Indeed,
Thomassen~\cite{2007CombC5} proved that \(\chithr(C_{m})=0\), while
Gishboliner, Jin and Sudakov~\cite{GishbolinerJinSudakov} determined
both removal thresholds of an odd cycle. Combining their results with
\cref{thm:main} gives the following strict chain.

\begin{cor}\label{cor:four-thresholds}
For every odd integer \(m\ge5\),
\[
0=\chithr(C_{m})
<\homthr(C_{m})
<\delta_{\textup{poly-rem}}(C_{m})=\frac{1}{m}
<\delta_{\textup{lin-rem}}(C_{m})=\frac{1}{4}.
\]
\end{cor}
Thus every odd cycle of length at least five answers the existential
second conjecture, while the whole family verifies the first one.
There is also a strict separation from the blowup threshold. Huang,
Liu, Rong and Xu~\cite{2025Blowup} proved that
\(\delta_{\textup{B}}(C_{m})=\frac{1}{m}\), and hence
\[
\homthr(C_{m})
<\delta_{\textup{B}}(C_{m})
=\delta_{\textup{poly-rem}}(C_{m}).
\]
In particular, \cref{thm:main} disproves their conjecture that
\(\homthr(H)=\delta_{\textup{B}}(H)\) for every graph \(H\)~\cite[Conjecture~5.1]{2025Blowup}.
To the best of our knowledge, odd
cycles form the first natural infinite family for which a strict
separation between the homomorphism and blowup thresholds is
established. 

There is also an asymmetric consequence. For two families of graphs \(\mathcal{F}_{1}\) and \(\mathcal{F}_{2}\), let \(\delta_{\textup{hom}}(\mathcal{F}_{1};\mathcal{F}_{2})\) be the minimum degree threshold in which the source graph is required to be \(\mathcal{F}_{1}\)-free and the bounded target is required to be \(\mathcal{F}_{2}\)-free. Gishboliner, Hurley and Wigderson~\cite[Theorem~1.3]{GHW2025} proved
that, for every odd \(m\ge3\),
\(
\homthr(\mathcal C_{m+4};\mathcal C_{m})=0,
\) and they conjectured that already
\(\homthr(\mathcal C_{m+2};\mathcal C_{m})=0\)
\cite[Conjecture~1.4]{GHW2025}. Our result identifies a different
phenomenon at the positive endpoint.
\begin{cor}\label{cor:asymmetric-threshold}
Let \(k\ge3\) and \(m=2k-1\). Then
\[
\frac{1}{2\left((k-1)^{4k-5}(2k-1)+\frac{(k-1)^{4k-5}-1}{k-2}\right)}
\le
\homthr(\mathcal C_{m};\{C_{m}\})
\le
\frac{2(m-1)}{2m(m-1)+1}
<
\frac{1}{m}
=
\homthr(\mathcal C_{m};\mathcal C_{m}).
\]
\end{cor}
The upper bound follows directly from \cref{thm:main}: every
\(\mathcal C_{m}\)-free graph is \(C_{m}\)-free, and the target
\(K_{m-1}\) supplied there is itself \(C_{m}\)-free. The strict
inequality is immediate, while the final equality is due to Ebsen and
Schacht~\cite{2020CombOdd}. The lower bound follows from the
strengthened form of our construction recorded in
\cref{rmk:lower-consequences}. Thus, while keeping the source family fixed, weakening the target condition from excluding every odd cycle of length at most \(m\) to
excluding only \(C_{m}\) strictly lowers the threshold, but does not
make it vanish.

Our methods also give new information about chromatic profiles. For a
graph or a family \(\mathcal F\) and an integer \(c\ge2\), let
\(\delta_{\chi}(\mathcal F,c)\) be the infimum of all \(\alpha\) such
that, for every \(\eps>0\), every sufficiently large
\(\mathcal F\)-free graph \(G\) with
\(\delta(G)\ge(\alpha+\eps)|V(G)|\) is \(c\)-colorable. Several parts
of the odd cycle profile are already known. For every odd \(m\ge5\),
the case \(c=2\) is completely determined: combining the work of Yuan
and Peng~\cite{yuan2024minimum} with the earlier small cycle cases
summarized in~\cite{YanPengYuan2024} gives
\(
\delta_{\chi}(C_{m},2)
=
\max\left\{\frac{1}{6},\frac{2}{m+2}\right\}.
\) For \(c\ge3\),
Yan, Peng and Yuan~\cite{YanPengYuan2024} proved
\(
\delta_{\chi}(C_{m},c)=\frac{1}{2c+2}
\)
whenever \(m=2k-1\) and \(k\ge3c+5\). Outside this long cycle range,
few exact values are known. Our arguments give the following bounds
in complementary parameter regimes.
\begin{theorem}\label{thm:chromatic-profile}
Let \(k\ge3\) and \(m=2k-1\). The following statements hold.
\begin{enumerate}
\item[\textup{(1)}] For every \(c\ge m-1\),
\(
\delta_{\chi}(C_{m},c)
\le\frac{2(m-1)}{2m(m-1)+1}.
\)
\item[\textup{(2)}] Let \(c\ge4\) and \(r=2\floor{\frac{c}{2}}\). Then
\(
\delta_{\chi}(\{C_{3},C_{5}\},c)
\le\frac{2r}{2r(r+1)+1}
\)
\item[\textup{(3)}] For every \(c\ge2\), 
\(
\delta_{\chi}(C_{m},c)
\ge\delta_{\chi}(\mathcal C_{m},c)
\ge\frac{1}{2\left((k-1)^{c-2}(2k-1)
+\frac{(k-1)^{c-2}-1}{k-2}\right)}.
\)
\end{enumerate}
\end{theorem}
Part~(1) follows immediately from \cref{thm:main}, which gives an \((m-1)\)-coloring above the stated density. Part~(2) can be easily obtained from~\cref{lem:upper-selector}. The construction proving part~(3) is described in
\cref{rmk:lower-consequences}. Notice that the exact result of Yan, Peng and Yuan~\cite{YanPengYuan2024} applies when the forbidden
cycle is long compared with the prescribed number of colors, whereas
\cref{thm:chromatic-profile}(1) treats the opposite range
\(c\ge m-1\). For the joint forbidden family
\(\{C_{3},C_{5}\}\), Thomassen's result~\cite{2007CombC5}
\(\delta_{\chi}(C_{5},c)\le\frac{6}{c}\) also gives
an upper bound of \(\frac{6}{c}\), while part~(2) improves the leading
constant from this comparison to \(1+o(1)\). Part~(3) provides a uniform construction, especially suited to diagonal regimes in which \(k\) and \(c\) grow together.

\section{The improved upper bound}\label{sec:upper}

In this section, we prove the upper bound in \cref{thm:main}.  Throughout
the section, fix an integer \(k\ge3\), put \(m=2k-1\), and note that
\(m\ge5\) is odd.  For a positive integer \(t\), write
\([t]=\{1,\ldots,t\}\).

The proof has two stages.  We first show that every vertex is incident
with only boundedly many edges lying in a short odd cycle.  After
deleting all these edges, we prove that the remaining spanning
subgraph still has enough minimum degree to force the ambient graph
to be \((m-1)\)-colorable.

We shall repeatedly work with a spanning subgraph \(F\) of an ambient
graph \(G\).  An \(F\)-edge is an edge belonging to \(E(F)\).  An
\(F\)-neighbor of a vertex \(v\) is a vertex joined to \(v\) by an
\(F\)-edge, the set of all such vertices is the \(F\)-neighborhood
\(N_{F}(v)\), and \(d_{F}(v)=|N_{F}(v)|\).  Thus a \(G\)-edge may or
may not be an \(F\)-edge.

We begin with the celebrated Erd\H{o}s--Gallai path theorem~\cite{1959ErdosGallai}.
\begin{lemma}[\cite{1959ErdosGallai}]\label{lem:upper-EG}
Let \(J\) be a graph on \(N\) vertices, and let \(\ell\ge1\).  If \(J\)
does not contain a path with \(\ell\) edges, then
\(e(J)\le\frac{(\ell-1)N}{2}\).
\end{lemma}

The next lemma gives the local estimate behind the deletion step.  If
the endpoints of a short odd path had many common neighbors, then
those common neighbors and the minimum degree condition would
produce a second path of exactly the complementary length.

\begin{lemma}\label{lem:upper-codegree}
Let \(n\ge2m^{2}(m+1)\), and let \(G\) be an \(n\)-vertex
\(C_{m}\)-free graph with
\(\delta(G)\ge\frac{2n}{2m+1}\).
If distinct vertices \(u,v\in V(G)\) are joined by an odd path of
length at most \(m-2\), then
\(|N_{G}(u)\cap N_{G}(v)|<m^{2}+m\).
\end{lemma}

\begin{proof}[Proof of Lemma~\ref{lem:upper-codegree}]
Let \(P\) be a path of odd length \(p\le m-2\) from \(u\) to \(v\).
If \(p=m-2\), then every common neighbor outside \(V(P)\) closes
\(P\) into a copy of \(C_{m}\).  Hence
\(|N_{G}(u)\cap N_{G}(v)|\le m-1<m^{2}+m\).

We can therefore assume that \(p\le m-4\).  Suppose for a contradiction
that \(|N_{G}(u)\cap N_{G}(v)|\ge m^{2}+m\).  Since \(P\) has at most
\(m-1\) vertices, we then choose a set
\[
S\subseteq\bigl(N_{G}(u)\cap N_{G}(v)\bigr)\setminus V(P)
\]
of size \(|S|=m^{2}\), and let
\(T=V(G)\setminus\bigl(S\cup V(P)\bigr)\). Notice that every vertex of \(S\)
has at most \(m^{2}-1\) neighbors in \(S\) and at most \(m-1\)
neighbors on \(P\), therefore by the
assumption \(n\ge2m^{2}(m+1)\), we have
\[
e(S,T)
\ge m^{2}\left(\delta(G)-m^{2}-m+2\right)
\ge m^{2}\left(\frac{2n}{2m+1}-m^{2}-m+2\right)>\frac{(m-2)n}{2}.
\]

The bipartite graph with parts \(S,T\) and edge set \(E_{G}(S,T)\)
has at most \(n\) vertices.  By \cref{lem:upper-EG}, the preceding
strict inequality gives an alternating path \(Q\) with \(m\) vertices,
or equivalently with \(m-1\) edges.

Because \(m\) and \(p\) are odd, \(m-p-2\) is even and lies between
\(2\) and \(m-3\).  Starting at the first vertex of \(Q\) if it lies
in \(S\), and at the second vertex otherwise, take a subpath \(Q'\)
with exactly \(m-p-2\) edges.  Its two endpoints lie in \(S\), and it
is disjoint from \(P\).  Join these endpoints to \(u\) and \(v\),
respectively, and then use \(P\) to return from \(v\) to \(u\).  The
resulting cycle has length \((m-p-2)+p+2=m\), a
contradiction.
\end{proof}

From now on, call an edge of a graph \(G\) \emph{bad} if it lies in an
odd cycle of length at most \(m\), and call it \emph{good} otherwise.
The next lemma is the bridge from the common degree estimate in Lemma \ref{lem:upper-codegree} to the later
structural argument: at the density needed for our theorem, every
vertex is incident with at most \(m\) bad edges.

\begin{lemma}\label{lem:upper-bad-degree}
Let \(G\) be an \(n\)-vertex \(C_{m}\)-free graph with
\(n\ge2m^{2}(m+1)\) and
\(\delta(G)\ge\frac{2(m-1)}{2m(m-1)+1}n\).  Suppose also that
\(
\frac{2m-3}{2m(m-1)+1}n
\ge\binom{m+1}{2}(m^{2}+m).
\)
Then every vertex of \(G\) is incident with at most \(m\) bad edges.
\end{lemma}

\begin{proof}[Proof of Lemma~\ref{lem:upper-bad-degree}]
Fix \(u\in V(G)\), and let
\(B_{u}=\{v\in N_{G}(u):uv\text{ is bad}\}\).

\begin{claim}\label{cl:upper-bad-path}
Every two distinct vertices of \(B_{u}\) are joined by an odd
path of length at most \(m-2\).
\end{claim}

\begin{poc}
Let \(x,y\in B_{u}\) be distinct.  Choose an odd cycle \(C\) containing
the edge \(ux\).  Its length is at most \(m-2\), because \(G\) is
\(C_{m}\)-free.  If \(y\notin V(C)\), take the \(x\)--\(u\) path
along \(C\) that avoids the edge \(ux\), and then append the edge
\(uy\).  The resulting \(x\)--\(y\) path is
odd, and has the same length as \(C\).  If \(y\in V(C)\), one
of the two \(x\)--\(y\) paths along \(C\) has odd length, and that
length is at most \(|C|-2\).  This proves the claim.
\end{poc}

Suppose that \(|B_{u}|\ge m+1\), and choose distinct vertices
\(v_{1},\ldots,v_{m+1}\in B_{u}\).  Since
\(\frac{2(m-1)}{2m(m-1)+1}>\frac{2}{2m+1}\),
\cref{lem:upper-codegree} and \cref{cl:upper-bad-path} give
\(|N_{G}(v_{i})\cap N_{G}(v_{j})|<m^{2}+m\) whenever \(i\neq j\).
The first two terms of inclusion--exclusion now give
\[
\left|\bigcup_{i=1}^{m+1}N_{G}(v_{i})\right|\ge\sum_{i=1}^{m+1}d_{G}(v_{i})
-\sum_{1\le i<j\le m+1}|N_{G}(v_{i})\cap N_{G}(v_{j})|>\frac{2(m+1)(m-1)}{2m(m-1)+1}n
-\binom{m+1}{2}(m^{2}+m)\ge n,
\]
a contradiction.  Therefore \(|B_{u}|\le m\).  Since \(u\) was
arbitrary, the lemma follows.
\end{proof}

The next lemma is the structural heart of the proof.  Its conclusion
concerns the chromatic number of the ambient graph \(G\), but its
degree condition concerns the spanning subgraph \(F\). 

\begin{lemma}\label{lem:upper-selector}
Let \(m\ge 5\) be odd and \(F\) be a spanning subgraph of a graph \(G\) on \(n\) vertices.
Suppose that no edge of \(F\) lies in a copy of \(C_{3}\) or \(C_{5}\)
in \(G\).  If
\(
\delta(F)>\frac{2(m-1)}{2m(m-1)+1}(n+1),
\)
then \(\chi(G)\le m-1\).
\end{lemma}

\begin{proof}[Proof of Lemma~\ref{lem:upper-selector}]
Suppose for a contradiction that \(\chi(G)\ge m\).  We first record
the observation that allows us to choose several disjoint
\(F\)-neighborhoods.

\begin{claim}\label{cl:upper-shared-neighborhood}
Suppose that \(A_{i}=N_{F}(s_{i})\) for some \(s_{i}\in V(G)\).  If
\(xy\in E(G)\), then there is no index \(i\) for which both
\(N_{F}(x)\cap A_{i}\) and \(N_{F}(y)\cap A_{i}\) are nonempty.
\end{claim}

\begin{poc}
Suppose otherwise, and choose vertices
\(a\in N_{F}(x)\cap A_{i}\) and \(b\in N_{F}(y)\cap A_{i}\).
The sequence \(xas_{i}byx\) is a closed walk of odd length five.
Every odd closed walk contains an odd cycle of no greater length, so
this walk contains a copy of \(C_{3}\) or \(C_{5}\).  All its edges
except possibly \(xy\) are \(F\)-edges.  The resulting odd cycle must
therefore contain an \(F\)-edge, contrary to the hypothesis.
\end{poc}

Choose a maximal pairwise disjoint collection of full
\(F\)-neighborhoods
\(
A_{i}=N_{F}(s_{i})\) with \(i\in[t].\)
By maximality, every vertex \(v\) has an \(F\)-neighbor in at least
one set \(A_{i}\), otherwise \(N_{F}(v)\) could be added to the
collection.  Choose \(c(v)\in[t]\) such that
\(N_{F}(v)\cap A_{c(v)}\neq\varnothing\).  By
\cref{cl:upper-shared-neighborhood}, the map \(c\) is a proper
\(t\)-coloring of \(G\), and hence \(t\ge m\).  On the other hand,
\(\frac{2(m-1)}{2m(m-1)+1}>\frac{1}{m+1}\), so the degree hypothesis
gives \(\delta(F)>\frac{n+1}{m+1}\).  Since the sets \(A_{i}\) are
pairwise disjoint and each has order at least \(\delta(F)\), we have
\(t\delta(F)\le n\), and therefore \(t\le m\).  Consequently,
\(t=m\).

Let \(R=V(G)\setminus\left(\bigcup_{i\in[m]}A_{i}\right)\).  Since
\(|A_{i}|\ge\delta(F)\) for every \(i\), counting all vertices gives
the important identity
\begin{equation}\label{eq:upper-defect}
    n-m\delta(F)
=|R|+\sum_{i\in[m]}\bigl(|A_{i}|-\delta(F)\bigr).
\end{equation}
In particular, \(n-m\delta(F)\ge0\).  Moreover,
\(\frac{2(m-1)}{2m(m-1)+1}>\frac{1}{m+1}\), and hence
\(|R|\le n-m\delta(F)<\delta(F)\).

For the remainder of the proof, write
\(d_{i}(v)=|N_{F}(v)\cap A_{i}|\), and define
\[
B_{i}=\bigg\{v\in V(G):d_{i}(v)>0\text{ and }d_{h}(v)=0
\text{ for every }h\in[m]\setminus\{i\}\bigg\}.
\]

\begin{claim}\label{cl:upper-singleton-crossing}
For every two distinct indices \(i,j\in[m]\), there is an edge of \(G\)
between \(B_{i}\) and \(B_{j}\).  Furthermore, every \(x\in B_{i}\)
satisfies \(d_{i}(x)\ge\delta(F)-|R|\) and
\[
|A_{i}\setminus N_{F}(x)|
\le|A_{i}|-\delta(F)+|R|
\le n-m\delta(F).
\]
\end{claim}

\begin{poc}
Suppose first that there is no \(G\)-edge between \(B_{i}\) and
\(B_{j}\).  Merge the colors \(i,j\) as follows.  Give the merged
color to a vertex if all the sets \(A_{h}\) containing one of its
\(F\)-neighbors have \(h\in\{i,j\}\).  Every other vertex has an
\(F\)-neighbor in some \(A_{h}\) with \(h\notin\{i,j\}\), we then choose one
such \(h\) as its color.  By \cref{cl:upper-shared-neighborhood}, two
adjacent vertices cannot receive the same unmerged color.  If two
adjacent vertices both received the merged color, the nonempty sets
of indices they see would be disjoint subsets of \(\{i,j\}\).  One
vertex would therefore belong to \(B_{i}\) and the other to \(B_{j}\),
contrary to the assumption.  We have constructed a proper
\((m-1)\)-coloring of \(G\), a contradiction.

Now let \(x\in B_{i}\).  Since \(x\) has no \(F\)-neighbor in any
\(A_{j}\) with \(j\neq i\), all its \(F\)-neighbors outside
\(A_{i}\) lie in \(R\).  Hence \(d_{i}(x)\ge\delta(F)-|R|\), and by \eqref{eq:upper-defect}, we have
\[
|A_{i}\setminus N_{F}(x)|
\le|A_{i}|-\delta(F)+|R|
\le n-m\delta(F).
\]
This proves the claim.
\end{poc}

\begin{claim}\label{cl:upper-heavy-unique}
No vertex \(v\in V(G)\) has two distinct indices \(i,j\in[m]\) for
which both \(d_{i}(v)>n-m\delta(F)+1\) and
\(d_{j}(v)>n-m\delta(F)+1\).
\end{claim}

\begin{poc}
Suppose that such \(v,i,j\) exist.  By
\cref{cl:upper-singleton-crossing}, choose an edge
\(xy\in E(G)\) with \(x\in B_{i}\) and \(y\in B_{j}\).  Since all
quantities are integers,
\[
|N_{F}(v)\cap N_{F}(x)\cap A_{i}|
\ge n-m\delta(F)+2+\delta(F)-|R|-|A_{i}|=2+\sum_{h\in[m]\setminus\{i\}}
\bigl(|A_{h}|-\delta(F)\bigr)\ge2,
\]
where the equality follows from \eqref{eq:upper-defect}.
The analogous intersection inside \(A_{j}\) also has at least two
vertices.  We may therefore choose
\(a\in N_{F}(v)\cap N_{F}(x)\cap A_{i}\setminus\{y\}\) and
\(b\in N_{F}(v)\cap N_{F}(y)\cap A_{j}\setminus\{x\}\).
The vertices \(v,x,y\) are distinct: \(v\) has \(F\)-neighbors in
both \(A_{i}\) and \(A_{j}\), whereas \(x\in B_{i}\) and
\(y\in B_{j}\), and \(x\neq y\) because \(xy\) is an edge.  Since
\(a\) is adjacent in \(F\) to both \(v\) and \(x\), it differs from
these two vertices. Similarly, \(b\) differs from \(v\) and \(y\).
The vertices \(a,b\) lie in the disjoint sets \(A_{i},A_{j}\), so
\(a\neq b\).  Finally, we selected \(a\neq y\) and \(b\neq x\).
Thus \(vaxybv\) is a pentagon in \(G\) containing \(F\)-edges,
a contradiction.
\end{poc}

We need one more local consequence of the assumption on \(F\).  It is
stated separately because it will be applied to paths whose middle
edge lies only in the ambient graph \(G\).

\begin{claim}\label{cl:upper-mixed-path}
Suppose that \(z,x,y,w\) are distinct vertices satisfying
\(zx,yw\in E(F)\) and \(xy\in E(G)\).
Then
\[
N_{F}(z)\cap N_{F}(w)=\varnothing.
\]
\end{claim}

\begin{poc}
Suppose that \(t\in N_{F}(z)\cap N_{F}(w)\).  Since the graphs are
loopless, \(t\notin\{z,w\}\).  If \(t=x\), then \(xywx\) is a triangle
containing the \(F\)-edges \(yw,wx\), if \(t=y\), then \(zxyz\) is a
triangle containing the \(F\)-edges \(zx,yz\).  If
\(t\notin\{x,y\}\), then \(zxywtz\) is a simple pentagon containing
edges of \(F\).  Every case
contradicts the hypothesis.
\end{poc}

We now use the degree condition to understand how the sets
\(A_{1},\ldots,A_{m}\) can be joined.  Rearranging the hypothesis of
the lemma gives the inequality
\begin{equation}\label{eq:upper-density}
    \delta(F)>2(m-1)\bigl(n-m\delta(F)+1\bigr).
\end{equation}
For distinct \(i,j\in[m]\), put
\[
Z_{i,j}=
\{v\in A_{i}:d_{j}(v)>n-m\delta(F)+1\}.
\]
Thus \(Z_{i,j}\) consists of the vertices of \(A_{i}\) having many
\(F\)-neighbors in \(A_{j}\).

Every \(A_{i}\) is independent even in \(G\).  Indeed, an edge
inside \(A_{i}=N_{F}(s_{i})\), together with \(s_{i}\), would form a
triangle containing edges of \(F\).  Thus \(d_{i}(v)=0\) for every
\(v\in A_{i}\).

Fix \(i\in[m]\) and \(v\in A_{i}\).  If
\(d_{j}(v)\le n-m\delta(F)+1\) for every
\(j\in[m]\setminus\{i\}\), then
\[
d_{F}(v)
\le|R|+(m-1)\bigl(n-m\delta(F)+1\bigr)
\le m\bigl(n-m\delta(F)+1\bigr)-1
<\delta(F),
\]
a contradiction, in particular, the last inequality follows from
\eqref{eq:upper-density}. Thus every vertex of \(A_{i}\) belongs to
at least one \(Z_{i,j}\).  By \cref{cl:upper-heavy-unique}, it belongs
to exactly one.  Hence the sets \(Z_{i,j}\),
\(j\in[m]\setminus\{i\}\), partition \(A_{i}\).

Call \(Z_{i,j}\) \emph{large} if
\(
|Z_{i,j}|\ge n-m\delta(F)+1.
\)

\begin{claim}\label{cl:upper-large-sets}
For every \(i\in[m]\), there exists \(j\in[m]\setminus\{i\}\) such
that \(Z_{i,j}\) is large.
Moreover, for distinct \(i,j\in[m]\) and
\(h\in[m]\setminus\{i,j\}\), the sets \(Z_{i,h}\) and \(Z_{j,h}\)
cannot both be large.
\end{claim}

\begin{poc}
Fix \(i\in[m]\). If none of the \(m-1\) sets \(Z_{i,j}\) were large,
then, since these sets partition \(A_{i}\) and
\(n-m\delta(F)\) is an integer,
\[
|A_{i}|\le(m-1)\bigl(n-m\delta(F)\bigr)<\delta(F),
\]
contrary to \(|A_{i}|\ge\delta(F)\).  The strict inequality follows
from \eqref{eq:upper-density}.

Suppose now that \(Z_{i,h}\) and \(Z_{j,h}\) are large for distinct
\(i,j\) and some \(h\in[m]\setminus\{i,j\}\). Choose an edge
\(xy\in E(G)\) with \(x\in B_{i}\) and \(y\in B_{j}\), as supplied by
\cref{cl:upper-singleton-crossing}. The vertex \(x\) misses at most
\(n-m\delta(F)\) vertices of \(A_{i}\), whereas
\(|Z_{i,h}|\ge n-m\delta(F)+1\). The analogous statement holds for
\(y\) and \(Z_{j,h}\). We may therefore choose
\(z\in Z_{i,h}\cap N_{F}(x)\) and
\(w\in Z_{j,h}\cap N_{F}(y)\).
The path \(zxyw\) is simple.  Indeed, the three edges of the path give
\(z\neq x\), \(x\neq y\), and \(y\neq w\), while \(z,w\) lie in the
disjoint sets \(A_{i},A_{j}\).  If \(z=y\), then
\(z\in Z_{i,h}\) would give \(d_{h}(y)>0\), contradicting
\(y\in B_{j}\) and \(h\neq j\). The case \(w=x\) is symmetric because
\(h\neq i\). This checks all six pairs among
\(z,x,y,w\).  By \cref{cl:upper-mixed-path},
\(
N_{F}(z)\cap N_{F}(w)=\varnothing.
\)

Since \(z\in A_{i}\), we have \(d_{i}(z)=0\). Moreover,
\(z\in Z_{i,h}\), so \cref{cl:upper-heavy-unique} gives
\(d_{\ell}(z)\le n-m\delta(F)+1\) for every
\(\ell\in[m]\setminus\{i,h\}\). At most \(|R|\) of its
\(F\)-neighbors lie in \(R\), and therefore
\[
d_{h}(z)
\ge\delta(F)-|R|-(m-2)\bigl(n-m\delta(F)+1\bigr).
\]
The same argument gives the corresponding bound for \(d_{h}(w)\).
Their \(F\)-neighborhoods inside \(A_{h}\) are disjoint, and hence
\[
2\left(\delta(F)-|R|-(m-2)\bigl(n-m\delta(F)+1\bigr)\right)
\le|A_{h}|.
\]
Rearranging gives the first inequality below. For the second, note
from \eqref{eq:upper-defect} that
\(2|R|+|A_{h}|-\delta(F)\le2(n-m\delta(F))\). Thus
\[
\begin{aligned}
\delta(F)
&\le2|R|+|A_{h}|-\delta(F)
+2(m-2)\bigl(n-m\delta(F)+1\bigr)\\
&\le2\bigl(n-m\delta(F)\bigr)
+2(m-2)\bigl(n-m\delta(F)+1\bigr)\\
&=2(m-1)\bigl(n-m\delta(F)+1\bigr)-2,
\end{aligned}
\]
contrary to \eqref{eq:upper-density}.
\end{poc}

\begin{claim}\label{cl:upper-permutation}
For every \(i\in[m]\), there is a unique \(f(i)\neq i\) for which
\(Z_{i,f(i)}\) is large, and the resulting map \(f:[m]\to[m]\) is a
permutation.
\end{claim}

\begin{poc}
Every \(i\) has at least one such index by
\cref{cl:upper-large-sets}.  If some \(i\) had two, choose one large
set \(Z_{h,j}\) for every \(h\neq i\).  We would obtain \(m+1\) large
sets with distinct second indices: the two chosen for \(i\) already
have different second indices, while \cref{cl:upper-large-sets}
prevents two different first indices from using the same second
index.  There are only \(m\) possible second indices, a
contradiction.  Thus \(f(i)\) is unique.  The same claim shows that
\(f\) is injective and hence a permutation.  Finally, \(f(i)\neq i\)
because it was chosen from \([m]\setminus\{i\}\).
\end{poc}

It remains to examine the cycle structure of \(f\).  Put
\(
X_{i}=Z_{i,f(i)}\) with \(i\in[m].\)
The other \(m-2\) parts of \(A_{i}\) have order at most
\(n-m\delta(F)\).  Hence, by \eqref{eq:upper-density},
\begin{equation}\label{eq:upper-X-size}
    |X_{i}|\ge|A_{i}|-(m-2)\bigl(n-m\delta(F)\bigr)\ge\delta(F)-(m-2)\bigl(n-m\delta(F)\bigr)>m\bigl(n-m\delta(F)+1\bigr)-1.
\end{equation}

Every \(x\in X_{i}\) is joined by an \(F\)-edge to almost every
vertex of \(X_{f(i)}\).  Indeed, among the sets \(A_{j}\), only
\(A_{f(i)}\) can contain more than \(n-m\delta(F)+1\)
\(F\)-neighbors of \(x\).  Therefore
\[
d_{f(i)}(x)
\ge\delta(F)-|R|-(m-2)\bigl(n-m\delta(F)+1\bigr).
\]
It follows from \eqref{eq:upper-defect} that the number of vertices of
\(A_{f(i)}\), and hence of \(X_{f(i)}\), missed by \(x\) is at most
\begin{equation*}
    |A_{f(i)}\setminus N_{F}(x)|
\le |R|+|A_{f(i)}|-\delta(F)
+(m-2)\bigl(n-m\delta(F)+1\bigr)\le(m-1)\bigl(n-m\delta(F)+1\bigr)-1.
\end{equation*}

\begin{claim}\label{cl:upper-cycle-shortening}
If \(f^{2}(i)\neq i\), then
\(
|X_{f(i)}|<|X_{i}|.
\)
\end{claim}

\begin{poc}
First, notice that every \(x\in X_{i}\) has at least
\[
|X_{f(i)}|-(m-1)\bigl(n-m\delta(F)+1\bigr)+1
\]
\(F\)-neighbors in \(X_{f(i)}\).  On the other hand, if
\(f^{2}(i)\neq i\), then a vertex \(y\in X_{f(i)}\) can have more
than \(n-m\delta(F)+1\) \(F\)-neighbors only in
\(A_{f^{2}(i)}\), not in \(A_{i}\).  Hence \(y\) has at most
\(n-m\delta(F)+1\) \(F\)-neighbors in \(X_{i}\subseteq A_{i}\).
Counting the \(F\)-edges between \(X_{i}\) and \(X_{f(i)}\) from both
sides gives
\[
|X_{i}|
\left(
|X_{f(i)}|-(m-1)\bigl(n-m\delta(F)+1\bigr)+1
\right)
\le\bigl(n-m\delta(F)+1\bigr)|X_{f(i)}|.
\]
Suppose for a contradiction that
\(|X_{f(i)}|\ge|X_{i}|\). Then from the above inequality we can see
\[
|X_{i}|
\left(
1-\frac{(m-1)\bigl(n-m\delta(F)+1\bigr)-1}{|X_{f(i)}|}
\right)
\le n-m\delta(F)+1.
\]
However, the left side is at least
\(
|X_{i}|-\left((m-1)\bigl(n-m\delta(F)+1\bigr)-1\right),
\)
which is strictly larger than \(n-m\delta(F)+1\) by
\eqref{eq:upper-X-size}, a contradiction.
\end{poc}

The permutation \(f\) has no fixed point.  If it had a cycle of length
at least three, then \cref{cl:upper-cycle-shortening}, applied at every
vertex around that cycle, would give a cyclic chain of strict
inequalities between the corresponding set sizes.  This is
impossible.  Hence every cycle of \(f\) has length two.  But a
permutation of the odd set \([m]\) cannot be a disjoint union of
transpositions.  This final contradiction proves
\(\chi(G)\le m-1\).
\end{proof}

With the structural lemma in hand, the upper bound follows by deleting
the bad edges.

\begin{proof}[Proof of the upper bound in Theorem~\ref{thm:main}]
Let \(\eps>0\), and let \(G\) be a \(C_{m}\)-free graph on \(n\)
vertices such that
\[
\delta(G)\ge
\left(\frac{2(m-1)}{2m(m-1)+1}+\eps\right)n.
\]
Take \(n\) sufficiently large that
\cref{lem:upper-bad-degree} applies and
\(
\eps n>
m+\frac{2(m-1)}{2m(m-1)+1}.
\)
Let \(F\) be the spanning subgraph obtained by deleting every bad edge
of \(G\).  By \cref{lem:upper-bad-degree}, every vertex loses at most
\(m\) incident edges. Hence
\[
\delta(F)
\ge\delta(G)-m\ge\left(\frac{2(m-1)}{2m(m-1)+1}+\eps\right)n
-m>\frac{2(m-1)}{2m(m-1)+1}(n+1).
\]
Every edge of \(F\) is good in \(G\).  In particular, no edge of
\(F\) lies in a copy of \(C_{3}\) or \(C_{5}\) in \(G\).  Therefore
\cref{lem:upper-selector} gives
\(\chi(G)\le m-1\), which also yields that \(G\hto K_{m-1}\). Since the target \(K_{m-1}\) has only
\(m-1\) vertices and is therefore \(C_{m}\)-free.  Thus, 
\[
\homthr(C_{m})
\le\frac{2(m-1)}{2m(m-1)+1}.
\]
This finishes the proof.
\end{proof}

\section{The lower bound}\label{sec:lower}
In this section, we prove the lower bound in \cref{thm:main}. 
We begin with listing some tools and known results used in the proof, together with the necessary notation. The odd girth of a graph \(G\) is the length of a shortest odd cycle in \(G\), and is infinite if \(G\) is bipartite.

The first auxiliary result lets us replace a fixed finite graph by an arbitrarily large graph of large odd girth, while preserving the existence of homomorphisms to all bounded targets. It is an immediate consequence of a result of
Ne\v{s}et\v{r}il and Zhu~\cite[Theorem~3.1]{2004JCTBZhu}, after taking sufficiently many
disjoint copies to make the graph large.

\begin{lemma}[\cite{2004JCTBZhu}]\label{lem:sparse}
Let \(H\) be a non-empty finite graph, and let \(m,L,n_{0}\) be positive integers. Then there is a graph \(W\) such that \(|V(W)|\ge n_0\), the odd girth of \(W\) is greater than \(L\), \(W\hto H\), and for every graph \(K\) with at most \(m\) vertices, \(W\hto K\) if and only if \(H\hto K\).
\end{lemma}

For an integer \(q\ge1\) and a graph \(G\), let \(M_q(G)\) be the following generalized Mycielski graph. Its vertices are
\[
        \{v^0,v^1,\ldots,v^q:v\in V(G)\}\cup\{w\}.
\]
Thus \(|M_q(G)|=(q+1)|V(G)|+1.\) Its edges are \(u^0v^0\) for every \(uv\in E(G)\), the edges \(u^iv^{i+1}\) and \(v^iu^{i+1}\) for every \(uv\in E(G)\) and \(0\le i<q\), and the root edges \(wv^q\) for every \(v\in V(G)\). The vertex \(w\) is often called the root. Furthermore, for positive integers \(q_1,\ldots,q_s\), write
\[
        M_{q_s}\cdots M_{q_1}(G)
        =
        M_{q_s}(M_{q_{s-1}}(\cdots M_{q_1}(G)\cdots)).
\]
When all \(q_i\) are equal to \(q\), write \(M_q^s(G)\) for this \(s\)-fold iteration, and put \(M_q^0(G)=G\). The next result says that a tower of generalized Mycielski steps starting from \(K_2\) has relatively large chromatic number. This is the Stiebitz's theorem~\cite{Stiebitz} in general, also see~\cite[Theorem~1.1]{2019EJC}.

\begin{lemma}[\cite{Stiebitz}]\label{lem:stiebitz}
For every sequence of positive integers \(q_1,\ldots,q_s\),
\[
        \chi(M_{q_s}\cdots M_{q_1}(K_2))\ge s+2.
\]
In particular, \(\chi(M_q^s(K_2))\ge s+2\) for all \(q\ge1\) and \(s\ge0\).
\end{lemma}

For a graph \(F\) and vertices \(u,v\in V(F)\), we write
\(\dist_F(u,v)\) for the length of a shortest path in \(F\) from \(u\) to \(v\),
where the length of a path is its number of edges. If no
path from \(u\) to \(v\) exists, we put \(\dist_F(u,v)=\infty\). In particular,
\(\dist_F(u,u)=0\). For \(j\in V(F)\) and an integer \(r\ge0\), define
\[
        B_F(j,r)
        :=
        F\big[\{u\in V(F):\dist_F(u,j)\le r\}\big].
\]
We call \(B_F(j,r)\) the ball of radius \(r\) around \(j\) in \(F\).
We shall use the following theorem of Corsini, Picasarri-Arrieta,
Pierron, Pirot and Robinson~\cite[Theorem~7.1]{CorsiniEtAl}.

\begin{lemma}[{\cite{CorsiniEtAl}}]\label{lem:ball-coloring}
Let \(L\ge2\), and let \(F\) be a graph containing no cycle of length
\(L+1\). Then, for every \(j\in V(F)\),
\[
        \chi\big(B_F(j,\floor{L/2})\big)\le 2L .
\]
\end{lemma}

For given integer \(k\ge 3,\) let \(m=2k-1\) and \(q=k-2\). For graphs \(H_{1},H_{2}\), we write \(H_{1}\nhto H_{2}\) if there is no graph homomorphism from \(H_{1}\) to \(H_{2}\). The following proposition is the reduction from finite obstructions to lower bound for homomorphism threshold. More precisely, it says that any finite graph \(T\) of odd girth \(m\), if it cannot map to a \(C_m\)-free graph, can be turned into dense \(C_m\)-free graphs with no bounded \(C_m\)-free homomorphic image.

\begin{prop}\label{prop:machine}
Let \(m\ge5\) be odd. Suppose that \(T\) is a non-empty finite graph whose odd girth is \(m\), and suppose \(T\nhto F\) for every \(C_m\)-free graph \(F\). Then \(\homthr(C_m)\ge\frac{1}{2|V(T)|}\). Moreover, for every pair of positive integers \(M,N_0\), there exists a \(C_m\)-free graph \(G\) with \(|V(G)|\ge N_0\), \(\delta(G)\ge \frac{|V(G)|}{2|V(T)|}\), and no homomorphism to any \(C_m\)-free graph on at most \(M\) vertices.
\end{prop}

\begin{proof}[Proof of Proposition~\ref{prop:machine}]
Let \(N=|V(T)|\). Apply \cref{lem:sparse} with \(H=T\), target order parameter \(M\), \(n_{0}=N_{0}\) and \(L=m\). We get a graph \(W_0\) with \(|V(W_0)|\ge N_0\), odd girth greater than \(m\), and a homomorphism \(\psi:W_0\hto T\). Moreover, \(W_0\) and \(T\) have the same homomorphisms to every graph on at most \(M\) vertices.

Write \(V(T)=\{t_1,\ldots,t_N\}\), and put \(R_i^0=\psi^{-1}(t_i)\). Each \(R_i^0\) forms an independent set, since otherwise an edge inside \(R_i^0\) would map to a loop at \(t_i\). Let \(r:=\max_i |R_i^0|\). Then we add appropriate isolated vertices to the smaller fibers so that all fibers have size \(r\). Call the new graph \(W\), and call the new fibers \(R_1,\ldots,R_N\). Also notice that, the map \(\psi\) still extends to a homomorphism \(W\hto T\), and adding isolated vertices creates no cycle.

Now add pairwise disjoint independent sets \(X_1,\ldots,X_N\), each of size \(r\), disjoint from \(V(W)\). Make \(X_i\) complete to \(R_i\), and add no other edges. Let the resulting graph be \(G\). Then \(|V(G)|=2Nr\). Moreover, every vertex in \(X_i\) has degree \(r\), because every vertex in \(R_i\) sees all vertices of \(X_i\), and the added vertices in the fibers, although isolated in \(W\), have degree \(r\) in \(G\). Hence \(\delta(G)\ge r=\frac{|V(G)|}{2N}\).

\begin{claim}\label{cl:machine-free}
The graph \(G\) is \(C_m\)-free.
\end{claim}

\begin{poc}
Suppose not, and let \(C\) be an \(m\)-cycle in \(G\). Since the odd girth of \(W\) is greater than \(m\), the cycle \(C\) cannot lie completely inside \(W\). Thus the cycle \(C\) contains at least one vertex from \(X:=X_1\cup\cdots\cup X_N\). Let \(a=|V(C)\cap X|\). Notice that \(X\) is an independent set, so no two of these \(a\) vertices are consecutive on \(C\). Hence \(1\le a\le\frac{m-1}{2}\).

Delete these \(a\) vertices from \(C\). What remains is a collection of paths whose edges all lie in \(W\). Project the remaining vertices to \(T\) using \(W\hto T\). If a deleted vertex \(x\in X_i\) had neighbors \(u,v\) on \(C\), then \(u,v\in R_i\), because \(x\) is adjacent only to \(R_i\). Thus both \(u\) and \(v\) project to the same vertex \(t_i\). In other words, after the projection to \(T\), the two ends created by deleting \(x\) can be glued back together at \(t_i\). Observe that doing this for every deleted vertex gives a closed walk in \(T\) of length \(m-2a\).

The length \(m-2a\) is positive, odd, and smaller than \(m\). Since our graphs are loopless, an odd closed walk of length \(1\) is impossible. If its length is at least \(3\), then it contains an odd cycle of length at most \(m-2a\), a contradiction to the fact that the odd girth of \(T\) is \(m\). Therefore \(G\) is \(C_m\)-free.
\end{poc}

\begin{claim}\label{cl:machine-no-target}
There is no \(C_m\)-free graph \(F\) on at most \(M\) vertices with \(G\hto F\).
\end{claim}

\begin{poc}
Suppose that such a graph \(F\) exists. Since \(W_0\) is an induced subgraph of \(G\), the homomorphism \(G\hto F\) restricts to a homomorphism \(W_0\hto F\). By the choice of \(W_0\) from \cref{lem:sparse}, the existence of a homomorphism \(W_0\hto F\) implies the existence of a homomorphism \(T\hto F\). This contradicts the assumption on \(T\), since \(F\) is \(C_m\)-free.
\end{poc}

The threshold statement now follows directly. The construction shows that for every \(M\) and \(N_0\) there is a \(C_m\)-free graph \(G\) with \(|V(G)|\ge N_0\), \(\delta(G)\ge \frac{1}{2N}\cdot |V(G)|\), and no homomorphism to any \(C_m\)-free graph on at most \(M\) vertices. Therefore \(\homthr(C_m)\ge \frac{1}{2N}\). This finishes the proof.
\end{proof}

It remains to build the finite obstruction required by \cref{prop:machine}. This is where the generalized Mycielski construction enters.

\begin{prop}\label{prop:Tk}
For every \(k\ge3\), there is a graph \(T_k\) whose odd girth is \(2k-1\),
\[
        |V(T_k)|=(k-1)^{4k-5}(2k-1)+\frac{(k-1)^{4k-5}-1}{k-2},
\]
such that \(T_k\nhto F\) for every \(C_{2k-1}\)-free graph \(F\).
\end{prop}

\begin{proof}[Proof of Proposition~\ref{prop:Tk}]
Put \(q=k-2\), \(h=k-1\), and \(s=4k-5\). We first record three simple facts about the graph \(M_q(G)\). Recall that \(M_q(G)\) has vertices \(v^0,v^1,\ldots,v^q\) for each \(v\in V(G)\), together with one extra vertex \(w\). Its edges are \(u^0v^0\) for every \(uv\in E(G)\), the edges \(u^iv^{i+1}\) and \(v^iu^{i+1}\) for every \(uv\in E(G)\) and every \(0\le i<q\), and the edges \(wv^q\) for every \(v\in V(G)\).

\begin{claim}\label{cl:mycielski-odd-girth}
If the odd girth of \(G\) is at least \(2k-1\), then the odd girth of \(M_q(G)\) is at least \(2k-1\). Moreover, if \(G\) contains a copy of \(C_{2k-1}\), then the odd girth of \(M_q(G)\) is exactly \(2k-1\).
\end{claim}

\begin{poc}
Let \(C\) be an odd cycle in \(M_q(G)\). We prove that \(|C|\ge2k-1\). First suppose that \(C\) does not contain \(w\). Replace each vertex \(v^i\) on \(C\) by \(v\). Every edge of \(C\) then becomes an edge of \(G\), because every edge of \(M_q(G)-w\) comes from an edge of \(G\). Thus \(C\) gives an odd closed walk in \(G\) of length \(|C|\). Every odd closed walk contains an odd cycle of length at most its own length. Since the odd girth of \(G\) is at least \(2k-1\), we get \(|C|\ge2k-1\).

Now suppose that \(C\) contains \(w\). The two neighbors of \(w\) on \(C\) have the form \(x^q\) and \(y^q\). After deleting \(w\), the remaining part of \(C\) is a path \(P\) from \(x^q\) to \(y^q\) in \(M_q(G)-w\). Since \(|C|\) is odd, \(|P|\) is also odd. Look at the superscripts along \(P\). Each edge of \(M_q(G)-w\) either has the form \(u^0v^0\), or changes the superscript by exactly \(1\). There is no edge \(u^iv^i\) with \(i>0\). If \(P\) used no edge of the form \(u^0v^0\), then the superscript would start at \(q\), end at \(q\), and change by \(1\) at every step, forcing \(|P|\) to be even. Hence \(P\) must use an edge \(u^0v^0\). To reach such an edge from superscript \(q\) takes at least \(q\) steps, the edge \(u^0v^0\) contributes one step, and returning to superscript \(q\) takes at least \(q\) more steps. Hence \(|P|\ge2q+1\), and therefore \(|C|=|P|+2\ge2q+3=2k-1\).

Thus the odd girth of \(M_q(G)\) is at least \(2k-1\). If \(G\) contains a copy of \(C_{2k-1}\), then the vertices \(v^0\) contain a copy of \(G\) inside \(M_q(G)\), so \(M_q(G)\) also contains a copy of \(C_{2k-1}\). Hence the odd girth is exactly \(2k-1\).
\end{poc}

Since \(M_q(K_2)\cong C_{2q+3}=C_{2k-1}\), define \(T_k=M_q^s(C_{2k-1})\). By \cref{cl:mycielski-odd-girth}, applied \(s\) times, the odd girth of \(T_k\) is exactly \(2k-1\). Indeed, every time we apply \(M_q\), the old graph remains inside the new graph as the copy formed by the vertices with superscript \(0\), so the old copy of \(C_{2k-1}\) is never lost.

We next define an auxiliary graph. Let \(F\) be a graph and let \(j\in V(F)\). Define \(A_{q,j}(F)\) as follows. Its vertices are all ordered tuples \((a_0,\ldots,a_q)\in V(F)^{q+1}\) such that \(a_q\in N_F(j)\). Two such tuples \(\boldsymbol a=(a_0,\ldots,a_q)\) and \(\boldsymbol b=(b_0,\ldots,b_q)\) are adjacent if \(a_0b_0\in E(F)\) and, for every \(0\le i<q\), both \(a_ib_{i+1}\in E(F)\) and \(b_ia_{i+1}\in E(F)\).

\begin{claim}\label{cl:local-graph}
For all graphs \(G,F\), \(M_q(G)\hto F\) if and only if there exists \(j\in V(F)\) such that \(G\hto A_{q,j}(F)\).
\end{claim}

\begin{poc}
Suppose first that \(\varphi:M_q(G)\hto F\). Let \(j=\varphi(w)\). For each \(x\in V(G)\), define \(f(x)=(\varphi(x^0),\ldots,\varphi(x^q))\). Since \(wx^q\in E(M_q(G))\), we have \(\varphi(x^q)\in N_F(j)\). Thus \(f(x)\) is a vertex of \(A_{q,j}(F)\). Now take an edge \(xy\in E(G)\). The edge \(x^0y^0\) gives \(\varphi(x^0)\varphi(y^0)\in E(F)\). For every \(0\le i<q\), the edges \(x^iy^{i+1}\) and \(y^ix^{i+1}\) give \(\varphi(x^i)\varphi(y^{i+1})\in E(F)\) and \(\varphi(y^i)\varphi(x^{i+1})\in E(F)\). These are exactly the conditions saying that \(f(x)\) and \(f(y)\) are adjacent in \(A_{q,j}(F)\). Hence \(f:G\hto A_{q,j}(F)\).

Conversely, suppose that \(f:G\hto A_{q,j}(F)\). For each \(x\in V(G)\), write \(f(x)=(a_0(x),\ldots,a_q(x))\). Define a map \(\Phi:V(M_q(G))\to V(F)\) by sending \(w\) to \(j\) and sending \(x^i\) to \(a_i(x)\). We check that \(\Phi\) preserves every edge of \(M_q(G)\). Since \(a_q(x)\in N_F(j)\), the edge \(wx^q\) is sent to an edge of \(F\). If \(xy\in E(G)\), then \(f(x)\) and \(f(y)\) are adjacent in \(A_{q,j}(F)\). Therefore \(a_0(x)a_0(y)\in E(F)\), and, for every \(0\le i<q\), both \(a_i(x)a_{i+1}(y)\in E(F)\) and \(a_i(y)a_{i+1}(x)\in E(F)\). Hence the edges \(x^0y^0\), \(x^iy^{i+1}\), and \(y^ix^{i+1}\) are all sent to edges of \(F\). Thus \(\Phi:M_q(G)\hto F\).
\end{poc}

\begin{claim}\label{cl:A-color}
If \(F\) is \(C_{2k-1}\)-free, then \(\chi(A_{q,j}(F))\le4k-4\) for every \(j\in V(F)\).
\end{claim}

\begin{poc}
Apply \cref{lem:ball-coloring} with \(L=2k-2\). Since \(\floor{L/2}=k-1=q+1\), the ball \(B_F(j,q+1)\) has a proper \((4k-4)\)-coloring. Fix such a coloring and call it \(c\).

We first show that if a tuple \(\boldsymbol a=(a_0,\ldots,a_q)\) has a neighbor in \(A_{q,j}(F)\), then \(\dist_F(a_0,j)\le q+1\). Let \(\boldsymbol b=(b_0,\ldots,b_q)\) be a neighbor of \(\boldsymbol a\). Define \(z_i=a_i\) when \(i\) is even, and define \(z_i=b_i\) when \(i\) is odd. From the definition of adjacency in \(A_{q,j}(F)\), the sequence \(z_0z_1\cdots z_q\) is a walk in \(F\). Its last vertex is either \(a_q\) or \(b_q\), and both \(a_q\) and \(b_q\) are adjacent to \(j\). Hence there is a walk from \(a_0\) to \(j\) of length at most \(q+1\), so \(\dist_F(a_0,j)\le q+1\).

Now color every tuple \((a_0,\ldots,a_q)\) which has a neighbor in \(A_{q,j}(F)\) by the color \(c(a_0)\). Tuples with no neighbor may be colored arbitrarily. If two tuples \((a_0,\ldots,a_q)\) and \((b_0,\ldots,b_q)\) are adjacent in \(A_{q,j}(F)\), then both \(a_0\) and \(b_0\) lie in \(B_F(j,q+1)\), and \(a_0b_0\in E(F)\). Since \(c\) is a proper coloring of \(B_F(j,q+1)\), we have \(c(a_0)\ne c(b_0)\). Thus this is a proper coloring of \(A_{q,j}(F)\), and \(\chi(A_{q,j}(F))\le4k-4\).
\end{poc}

We now prove that \(T_k\) cannot map to any \(C_{2k-1}\)-free graph. Suppose for a contradiction that \(T_k\hto F\), where \(F\) is \(C_{2k-1}\)-free. Since \(T_k=M_q(M_q^{s-1}(C_{2k-1}))\), \cref{cl:local-graph} gives a vertex \(j\in V(F)\) such that \(M_q^{s-1}(C_{2k-1})\hto A_{q,j}(F)\). But \(C_{2k-1}\cong M_q(K_2)\), so \(M_q^{s-1}(C_{2k-1})\cong M_q^s(K_2)\). By \cref{lem:stiebitz}, we have \(\chi(M_q^s(K_2))\ge s+2=4k-3\). On the other hand, \cref{cl:A-color} gives \(\chi(A_{q,j}(F))\le4k-4\). This is impossible, because a homomorphism \(X\hto Y\) implies \(\chi(X)\le\chi(Y)\). Applying this to \(M_q^s(K_2)\hto A_{q,j}(F)\) gives \(4k-3\le\chi(M_q^s(K_2))\le\chi(A_{q,j}(F))\le4k-4\), a contradiction. Therefore \(T_k\nhto F\) for every \(C_{2k-1}\)-free graph \(F\).

It remains to count the vertices. If \(G\) has \(n\) vertices, then \(M_q(G)\) has \((q+1)n+1=hn+1\) vertices. Starting from \(C_{2k-1}\), which has \(2k-1\) vertices, and applying \(M_q\) exactly \(s\) times, we get
\[
        |V(T_k)|=h^s(2k-1)+(1+h+\cdots+h^{s-1}).
\]
Since \(h=k-1\) and \(s=4k-5\), this gives
\[
        |V(T_k)|=(k-1)^{4k-5}(2k-1)+\frac{(k-1)^{4k-5}-1}{k-2}.
\]
This completes the proof.
\end{proof}

\begin{proof}[Proof of the lower bound in Theorem~\ref{thm:main}]
Apply \cref{prop:machine} with \(m=2k-1\) and \(T=T_k\), where \(T_k\) is given by \cref{prop:Tk}. The count in \cref{prop:Tk} gives
\[
        \homthr(C_{2k-1})\ge
        \frac{1}{2\left((k-1)^{4k-5}(2k-1)+\frac{(k-1)^{4k-5}-1}{k-2}\right)}.
\]
This proves the theorem.
\end{proof}
\begin{rmk}\label{rmk:lower-consequences}
We remark that there are two further consequences of the preceding construction.
First observe that the graph produced in the proof of
\cref{prop:machine} actually has odd girth greater than \(m\). Indeed,
if it contains an odd cycle \(C\) of length \(\ell\le m\), let \(a\)
be the number of vertices of \(C\) lying in the newly added
independent sets \(X_{i}\). Since the sparse core has odd girth greater
than \(m\), we have \(a\ge1\), while independence gives
\(a\le\frac{\ell-1}{2}\). Deleting these vertices, projecting the
remaining paths to \(T\), and gluing the two ends corresponding to
each deleted vertex gives an odd closed walk in \(T\) of positive
length \(\ell-2a<m\), contradicting the odd girth of \(T\).

This strengthening proves the lower bound in
\cref{cor:asymmetric-threshold}. Fix \(k\ge3\), put \(m=2k-1\), and
take \(T=T_{k}\) from \cref{prop:Tk}. For every pair of positive
integers \(M,N_{0}\), the construction gives an
\(\mathcal C_{m}\)-free graph \(G\) of order at least \(N_{0}\) with
\(\delta(G)\ge\frac{|V(G)|}{2|V(T_{k})|}\). If \(G\hto F\) for a
\(C_{m}\)-free graph \(F\) of order at most \(M\), then restricting
to the sparse core and applying \cref{lem:sparse} would give
\(T_{k}\hto F\), contrary to \cref{prop:Tk}. Since both \(M\) and
\(N_{0}\) are arbitrary, the definition of the asymmetric
homomorphism threshold gives the required lower bound.

The same construction with a shorter Mycielski tower gives
\cref{thm:chromatic-profile}(3). For \(c\ge2\), put \(q=k-2\) and
\(T_{k,c}=M_{q}^{c-2}(C_{m})\). Then \(T_{k,c}\) has odd girth \(m\).
Moreover, \(C_{m}\cong M_{q}(K_{2})\), so
\(T_{k,c}\cong M_{q}^{c-1}(K_{2})\) and
\(\chi(T_{k,c})\ge c+1\) by \cref{lem:stiebitz}. Its order is
\[
|V(T_{k,c})|
=(k-1)^{c-2}(2k-1)
+\frac{(k-1)^{c-2}-1}{k-2}
=N_{k,c}.
\]
For every \(N_{0}\), repeat the construction in
\cref{prop:machine}, applying \cref{lem:sparse} with
\(H=T_{k,c}\), target-size parameter \(c\), \(L=m\), and
\(n_{0}=N_{0}\). The resulting graph \(G\) has odd girth greater than
\(m\), satisfies
\(\delta(G)\ge\frac{|V(G)|}{2N_{k,c}}\), and is not \(c\)-colorable:
otherwise \(G\hto K_{c}\), and restricting to the sparse core would
give \(T_{k,c}\hto K_{c}\). Since \(N_{0}\) is arbitrary,
\(
\delta_{\chi}(C_{m},c)
\ge\delta_{\chi}(\mathcal C_{m},c)
\ge\frac{1}{2N_{k,c}},
\)
as claimed.
\end{rmk}

\section{Some remarks and related consequences}
The central result of this paper is that, for every \(k\ge3\), the
homomorphism threshold of \(C_{2k-1}\) is strictly smaller than
\(\frac{1}{2k-1}\). In view of the considerable evidence that had
pointed towards equality~\cite{2020CombOdd,2025Blowup,2019JGTC3C5}, this is a qualitative change in our
understanding of the problem rather than merely a numerical
improvement. On the lower bound side, we give a new, entirely
graph-theoretic construction establishing a positive quantitative
lower bound.  Taken together, these results suggest that determining
the homomorphism thresholds of odd cycles may be substantially more
difficult than previously expected.

Our upper bound argument gives
\[
\homthr(C_{2k-1})
\le\frac{4(k-1)}{4(k-1)(2k-1)+1}
=\frac{1}{2k-1}-\eps_{k},
\]
where
\(
\eps_{k}
=\frac{1}{(2k-1)\bigl(4(k-1)(2k-1)+1\bigr)}.
\)
We do not expect this correction term to be optimal.  It seems
plausible that a further development of the structural ideas in our
upper bound proof will produce a substantially larger value of
\(\eps_{k}\).  At the same time, our results make it difficult to
formulate a convincing conjecture for the exact value of
\(\homthr(C_{2k-1})\). The known upper bound already implies that
\(\homthr(C_{2k-1})\) tends to zero as \(k\) tends to infinity.  What
remains unclear is its correct scale: we cannot presently decide
whether it remains of order \(\frac{1}{2k-1}\), or whether it is
asymptotically smaller by an unbounded factor.  We suspect the latter.

\begin{conjecture}\label{conj:asymptotic-threshold}
\(
\lim_{k\to\infty}(2k-1)\homthr(C_{2k-1})=0.
\)
\end{conjecture}
We next record a structural consequence of the upper bound proof of Theorem~\ref{thm:main}. Call
an edge of \(G\) short-odd if it lies in an odd cycle of length at most \(m\), and let \(B_{m}(G)\) be the spanning subgraph formed by
all short-odd edges.

\begin{cor}\label{cor:bounded-degree-deletion}
Let \(m\ge 5\) be an odd integer. For every \(\eps>0\) and every sufficiently large \(C_{m}\)-free
\(n\)-vertex graph \(G\) satisfying
\(
\delta(G)\ge(\frac{2(m-1)}{2m(m-1)+1}+\eps)n,
\)
we have \(\Delta(B_{m}(G))\le m\). Consequently, the spanning graph
\(F=G-E(B_{m}(G))\) has odd girth greater than \(m\) and satisfies
\(\delta(F)\ge\delta(G)-m\) and
\(e(B_{m}(G))\le\frac{mn}{2}\).
Moreover, \(E(B_{m}(G))\) is the union of at most \(m+1\) matchings.
\end{cor}

\begin{proof}[Proof of Corollary~\ref{cor:bounded-degree-deletion}]
For sufficiently large \(n\), the hypotheses of
\cref{lem:upper-bad-degree} hold. The bad edges in that lemma are
exactly the edges of \(B_{m}(G)\), so
\(\Delta(B_{m}(G))\le m\). Deleting this graph removes at most \(m\)
edges at every vertex, and hence
\(\delta(F)\ge\delta(G)-m\). Every odd cycle of length at most \(m\)
in \(F\) would consist entirely of short-odd edges of \(G\), all of
which were deleted. Thus the odd girth of \(F\) is greater than \(m\).
The bound on \(e(B_{m}(G))\) follows from the handshaking lemma.
Finally, Vizing's theorem gives
\(\chi'(B_{m}(G))\le\Delta(B_{m}(G))+1\le m+1\), which is precisely
the asserted decomposition into matchings.
\end{proof}

Our new construction has some other consequences. For example, although the main results of this paper concern connected
forbidden graphs, the same construction also produces disconnected
graphs for which the chromatic and homomorphism thresholds differ.
If \(t\ge1\), write \(tK_{1}\) for the graph consisting of \(t\)
isolated vertices.  We consider
\(D_{m,t}=C_{m}\sqcup tK_{1}\), the disjoint union of an odd cycle and
\(t\) isolated vertices.

\begin{prop}\label{prop:cycle-plus-isolates}
Let \(m=2k-1\ge5\), let \(t\ge1\), put \(h=k-1\), and let \(s=\max\{2m-3,m+t-3\}\). Then \(\chithr(D_{m,t})=0<\homthr(D_{m,t})\), and
\[
        \homthr(D_{m,t})\ge
        \frac{1}{2\left(h^s m+\frac{h^s-1}{h-1}\right)}.
\]
\end{prop}

\begin{proof}[Proof of Proposition~\ref{prop:cycle-plus-isolates}]
The equality \(\chithr(D_{m,t})=0\) follows from the classification of chromatic thresholds~\cite{2013Adv} since \(D_{m,t}\) is near-acyclic and has chromatic number \(3.\)

We now prove the lower bound for \(\homthr(D_{m,t})\). Put \(q=k-2\) and \(T=M_q^s(C_m)\). The same argument as in the proof of \cref{prop:Tk} shows that the odd girth of \(T\) is \(m\). It also shows that \(T\nhto F\) for every \(C_m\)-free graph \(F\), otherwise we have \(M_q^{s-1}(C_m)\hto A_{q,j}(F)\) for some \(j\), while \(M_q^{s-1}(C_m)=M_q^s(K_2)\) has chromatic number at least \(s+2\) by Lemma~\ref{lem:stiebitz} and \(\chi(A_{q,j}(F))\le2m-2\) by \cref{cl:A-color}, which contradicts \(s\ge2m-3\). We shall also use that \(\chi(T)\ge s+3\ge m+t\), since \(T=M_q^{s+1}(K_2)\).

Apply the construction from \cref{prop:machine} with this graph \(T\). For every pair of positive integers \(M,N_0\), it gives a \(C_m\)-free graph \(G\), hence a \(D_{m,t}\)-free graph, with \(|V(G)|\ge N_0\) and \(\delta(G)\ge \frac{|V(G)|}{2|V(T)|}\). It remains to see that \(G\) has no homomorphism to any \(D_{m,t}\)-free graph on at most \(M\) vertices. Suppose otherwise, and let \(G\hto F\), where \(F\) is \(D_{m,t}\)-free and \(|V(F)|\le M\). As in \cref{prop:machine}, we have \(T\hto F\) and \(F\) contains a copy of \(C_m\). Since \(F\) is \(D_{m,t}\)-free, there are at most \(t-1\) vertices outside this cycle, otherwise the cycle together with any \(t\) outside vertices would contain \(C_m\sqcup tK_1\). Thus \(|V(F)|\le m+t-1\), so \(\chi(F)\le m+t-1\), contradicting \(T\hto F\) and \(\chi(T)\ge m+t\).

Finally, \(|V(T)|=h^s m+(1+h+\cdots+h^{s-1})=h^s m+\frac{h^s-1}{h-1}\). The same threshold argument as in \cref{prop:machine} gives the stated bound.
\end{proof}

\section*{Acknowledgment}
Zixiang Xu would like to thank Dr. Mingyuan Rong for many inspiring discussions on constructions for the problems of blowup thresholds and homomorphism thresholds over the past two years.

\bibliographystyle{abbrv}
\bibliography{HomoC2k-1}

\end{document}